\documentclass[reqno,11pt,centertags,draft]{amsart}
\usepackage{amsmath,amsthm,amscd,amssymb,latexsym,upref}

\newcommand{\bd}{{\mathbb{D}}}
\newcommand{\bn}{{\mathbb{N}}}

\newcommand{\bc}{{\mathbb{C}}}

\newcommand{\bt}{{\mathbb{T}}}

\renewcommand{\a}{\alpha}
\renewcommand{\b}{\beta}
\renewcommand{\l}{\lambda}

\newcommand{\s}{\sigma}

\newcommand{\ep}{\varepsilon}

\newcommand{\p}{\varphi}

\renewcommand{\th}{\theta}

\newcommand{\oo}{\Omega}

\newcommand{\z}{\zeta}

\newcommand{\ovl}{\overline}

\newcommand{\lp}{\left(}
\newcommand{\rp}{\right)}
\newcommand{\lt}{\left}
\newcommand{\rt}{\right}

\allowdisplaybreaks
\numberwithin{equation}{section}

\newtheorem{theorem}{Theorem}[section]
\newtheorem{lemma}[theorem]{Lemma}

\newtheorem{proposition}[theorem]{Proposition}
\theoremstyle{definition}
\newtheorem{definition}[theorem]{Definition}
\newtheorem{remark}[theorem]{Remark}

\newtheorem{example}[theorem]{Example}

\begin{document}

\title[Resolvent of Toeplitz operators]
%% \title[Resolvent growth of Toeplitz operators] ??
{On the growth of resolvent of Toeplitz operators}
%% 
% if one speaks of a “resolvent growth”, it’s always “closely to the spectrum”, it
% doesn’t seem reasonable to speak on  “resolvent growth” far from spectrum...
%
% that’s why i’d suggest to change the title
%{On the growth of resolvent of Toeplitz operators with Laurent polynomial symbols of certain classes}

\author{L. Golinskii}
\address{B. Verkin  Institute for Low Temperature Physics and Engineering 
of the National Academy of Sciences of Ukraine, 47 Nauky ave., Kharkiv 61103, Ukraine}
\email{leonid.golinskii@gmail.com; golinskii@ilt.kharkov.ua}

\author{S. Kupin}
\address{IMB, CNRS, Universit\'e de Bordeaux, 351 ave. de la Lib\'eration, 33405 Talence Cedex, France}
\email{skupin@math.u-bordeaux.fr}

\author{A. Vishnyakova}
\address{Holon Institute of Technology, 
52 Golomb Street, POB 305 Holon 5810201, Israel}
\email{annalyticity@gmail.com}

\begin{abstract}
%\if{Recall that a Laurent polynomial on the complex plane is a function of the form
%$$ b(z)=\sum^k_{j=-m}b_jz^j, \quad b_{-m}b_k\not=0.
%$$
%In this note, we study the resolvent growth a Toeplitz operator defined on Hardy space having Laurent-polynomial symbol satisfying certain regularity conditions.}\fi 
%defined on the classical Hardy space $H^2(\bt)$.
We study the growth of the resolvent of a Toeplitz operator $T_b$, defined on the Hardy space, in terms of the distance to its spectrum $\s(T_b)$. We are primarily interested in
the case when the symbol $b$ is a Laurent polynomial (\emph{i.e., } the matrix $T_b$ is banded). We show that for an arbitrary such symbol the growth of the resolvent
is quadratic \eqref{quadresgr}, and under certain additional assumption it is linear \eqref{linresgr}. We also prove the quadratic growth of the resolvent for a certain class
of non-rational symbols.
\end{abstract}


\maketitle

\section*{Introduction}

Let $T$ be a bounded linear operator on a Hilbert space. The complex number $\l$ is said to lie in the resolvent set $\rho(T)$ of $T$ if
there exists a bounded inverse operator $(T-\l)^{-1}$. The latter is known as the resolvent of $T$. The spectrum $\s(T)$ of $T$ is by definition
the complement of the resolvent set. The spectrum of $T$ is known to be a compact set on the complex plane. Recall also that the essential spectrum
of $T$, $\s_{ess}(T)\subset\s(T)$, is the set of complex numbers $\l$ so that $T-\l$ is not %even 
a Fredholm operator.

The spectrum and resolvent are the key concepts of the whole operator theory. An enormous amount of effort has been devoted to their study for
over a hundred years. In particular, the study of the behavior of the resolvent norm $\|(T-\l)^{-1}\|$, as $\l$ varies over the resolvent set,
has attracted much attention. Such norm grows unboundedly as the point $\l$ approaches the spectrum $\s(T)$. What is more to the point, there is
a universal below bound
\begin{equation}\label{belbou}
\|(T-\l)^{-1}\|\ge \frac1{{\rm dist}\,\bigl(\l,\s(T)\bigr)}\,, \qquad \l\in\rho(T),
\end{equation}
which actually holds for any closed operator on a Hilbert space.

In this paper we focus on upper bounds for the norm of resolvent of the form
\begin{equation}\label{upbou}
\|(T-\l)^{-1}\|\le \Phi\bigl({\rm dist}\,\bigl(\l,\s(T)\bigr)\bigr)\,, \qquad \l\in\rho(T),
\end{equation}
for particular functions $\Phi$ on the positive half-line, see variety of examples of such bounds in \cite[Section 5]{fago15} and references
therein. Apart from an intrinsic beauty, such bounds appear in at least two problems of analysis. The first one concerns the existence of 
hyperinvariant subspaces for a linear operator $T$, that is, subspaces invariant for any operator commuting with $T$. The key tool here is a
celebrated result of Lyubich and Matsaev, which relies heavily on bounds of the form \eqref{upbou}, see \cite{pel}. The second one is the problem from perturbation
theory related to the discrete spectrum of the trace (or general Schatten-von Neumann) class perturbation of certain operators, see, e.g., \cite{fago15, bcgk}. 
The point is that such discrete spectrum (counting algebraic multiplicities) agrees with the zero set (counting multiplicities) of a certain analytic function, 
known as the perturbation determinant. The bound \eqref{upbou} implies a similar bound for the perturbation determinant, and one obtains the so-called 
generalized Blaschke-type conditions for its zero set, see 
recent results by Borichev--Golinskii--Kupin \cite{bgk1, bgk2}, Favorov--Golinskii \cite{fago09, fago12, fago15}, and others.

In this paper we deal with two cases of $\Phi$
$$ \Phi(x)=\Phi_l(x):=\frac{C}{x}, \qquad \Phi(x)=\Phi_q(x):=\frac{C}{x}\Bigl(1+\frac1{x}\Bigr), $$
where $C$ is a positive constant depending on $T$. Note that, in view of \eqref{belbou}, bound \eqref{upbou} with $\Phi=\Phi_l$ is optimal. 
We refer to the first case as a {\sl Linear Resolvent Growth} (LRG) condition, see \eqref{linresgr} and also Benamara--Nikolski \cite{benk}, 
Nikolski--Treil \cite{nitr}, and to the second one as a {\sl Quadratic Resolvent Growth} (QRG) condition, see \eqref{quadresgr}.

Toeplitz operators have been enjoying immense popularity for many de\-ca\-des. They appear in a number of problems in different areas of analysis.
They give rise to interesting and difficult problems, and lead to beautiful results; see Grenander--Szeg\H o \cite{grsz} and Nikolski \cite{nik} 
for an account on the topic. Given a function $b\in L^\infty(\bt)$, the functional realization
of a Toeplitz operator $T_b$ on $H^2(\bt)$ is
$$  (T_b f)(t)=P_+\bigl(b(t)f(t)\bigr), \qquad f\in H^2(\bt), 
$$
where $P_+$ is the orthogonal projection from $L^2(\bt)$ onto $H^2(\bt)$, $b$ is a symbol of $T_b$. The matrix realization of the operator is
$$
T_b=\|b_{i-j}\|_{i,j\ge1}: \ell^2(\bn)\to \ell^2(\bn),
$$ 
or, in words, this is a semi-infinite matrix with constant diagonals. Here $b_n$ are the Fourier coefficients of the symbol $b$.

Among the results concerning the resolvent growth for general Toeplitz operators, let us mention the following one, which states that the LRG condition holds 
for $T_b$ as long as $\s(T_b)$ is a convex set (for an extended version of this result see \cite[Theorem 4.29]{bogr}).

We say that a Toeplitz operator is \emph{banded}, iff $b_n=0$ for all $|n|$ large enough.
It is clear that these are Toeplitz operators  with symbols being Laurent polynomials
\begin{equation*}%\label{laur}
b(z):=\frac{b_{-m}}{z^m}+\ldots+b_0+\ldots+b_kz^k, \quad m,k\in\bn, \quad b_{-m}b_k\not=0.
\end{equation*}

In the paper we focus on two classes of such symbols. The class of Laurent polynomials with Jordan property (see Definition \ref{deflj}) is defined in terms
of the image $b(\bt)$ of the unit circle. The class of regular Laurent polynomials (see Definition~\ref{defr}) is defined in terms of the zeros
of the algebraic equation
$$ z^m\bigl(b(z)-w\bigr)=b_{-m}+\ldots+(b_0-w)z^m+\ldots+b_kz^{m+k}=0. 
$$
Section 1 of the paper is devoted to the discussion of certain properties of Laurent polynomials with Jordan property and regular Laurent polynomials.  
For instance, the former is a proper subclass of the latter, see Theorem \ref{ljreg}. Our main result, obtained in Section 2, states that the LRG 
condition \eqref{linresgr} holds as long as $b$ is a regular Laurent polynomial, see Theorem \ref{resol}. In the final Section \ref{s3}, we study the class of (non-rational)
symbols subject to \eqref{addas} and show that the corresponding Toeplitz operators $T_b$ obey the QRG condition, see Theorem~\ref{infcoef}.

The class of symbols $b$ for which the bound \eqref{upbou} holds with $T=T_b$ (and a single function $\Phi$ within the whole class) is small enough. By the
result of Treil \cite{tr}, given a sequence $\{\l_n\}$ of points on the unit disk, and a sequence $\{A_n\}$ of positive numbers, $A_n\to\infty$, there is a continuous
and unimodular symbol $b$ so that
$$ \|(T_b-\l_n)^{-1}\|>A_n. $$
So, \eqref{upbou} is certainly false for the class of all continuous symbols. 

As usual we let $\bt=\{z: |z|=1\}$ to be the unit circle of the complex plane, and $\bd=\bd_+=\{z: |z|<1\}, \  \bd_-=\{z: |z|>1\}$.

\subsection*{Acknowledgements} LG thanks S. Favorov for helpful discussions regarding the subject of Section \ref{s3}.

\section{Jordan property and regular Laurent polynomials}

For now, the main object under consideration is a Laurent polynomial
\begin{equation}\label{laur}
b(z):=\frac{b_{-m}}{z^m}+\ldots+b_0+\ldots+b_kz^k, \quad m,k\in\bn, \quad b_{-m}b_k\not=0.
\end{equation}

\subsection{Laurent polynomials with Jordan property}

\begin{definition}\label{deflj}
Let $b(\bt)$ be the image of the unit circle $\bt$ under the mapping $b$, a closed curve on the complex plane.
A Laurent polynomial $b$ is said to \emph{possess the Jordan property} (or, in short, $b$ is a \emph{LJ-polynomial}) if $b(\bt)$ is 
a Jordan curve (no self-intersections), and $b'(t)\not=0$, $t\in\bt$ (no cusps).
\end{definition}

It is quite unlikely that there exist transparent necessary and sufficient conditions in terms of the Laurent coefficients 
for a polynomial $b$ to be LJ-polynomial. Here is a simple sufficient one.

\begin{proposition}\label{ljcoef}
Assume that $\max(m,k)\ge2$, and one of the two conditions holds
\begin{equation}\label{coef1}
|b_{\pm1}|\ge |b_{\mp1}|+\sum_{j=2}^k j|b_j|+\sum_{j=2}^m j|b_{-j}|.
\end{equation}
Then $b(\bt)$ is a Jordan curve. Moreover, if one of the strict inequalities holds
\begin{equation}\label{coef2}
|b_{\pm1}|> |b_{\mp1}|+\sum_{j=2}^k j|b_j|+\sum_{j=2}^m j|b_{-j}|,
\end{equation}
then $b$ is a LJ-polynomial.
\end{proposition}
\begin{proof}
Suppose, on the contrary, 
%% ??
that there are two different points $\z_1\not=\z_2$ on $\bt$, so that $b(\z_1)=b(\z_2)$. Then
$$ b_1(\z_1-\z_2)=\sum_{j\not=1} b_j(\z_2^j-\z_1^j)=b_{-1}\,\frac{\z_1-\z_2}{\z_1\z_2}+\sum_{|j|\ge2} b_j(\z_2^j-\z_1^j). $$
Assume that \eqref{coef1} holds with the upper sign.
For $\z_1\not=\z_2$ on $\bt$ one has
$$ \left|\frac{\z_1^j-\z_2^j}{\z_1-\z_2}\right|<|j|, \qquad |j|\ge2, $$
and we come to the inequality
$$ |b_1|<|b_{-1}|+\sum_{j=2}^k j|b_j|+\sum_{j=2}^m j|b_{-j}|, $$
which contradicts \eqref{coef1}. The case of the lower sign in \eqref{coef1} is the same, so the first statement is proved.

Concerning the second one, we write
\begin{equation*}
\begin{split}
b'(\z) &=\sum_{j=1}^k jb_j\z^{j-1}-\sum_{j=1}^m jb_{-j}\z^{-j-1} \\
&=b_1-b_{-1}\z^{-2}+\sum_{j=2}^k jb_j\z^{j-1}-\sum_{j=2}^m jb_{-j}\z^{-j-1}.
\end{split}
\end{equation*}
If $b'(\z_0)=0$ for some $\z_0\in\bt$, then
$$ |b_{\pm1}|\le |b_{\mp1}|+\sum_{j=2}^k j|b_j|+\sum_{j=2}^m j|b_{-j}|, $$
in contradiction with \eqref{coef2}, and the second statement follows.
\end{proof}

The case $m=k=1$ will be discussed later in Example \ref{laujac}.

\begin{example}
Let $b_1(z)=nz^{-1}-z^n$, so that \eqref{coef1} holds, and there is no self-intersection in $b_1(\bt)$. On the other hand,
$$ b_1'(z)=-\frac{n(z^{n+1}+1)}{z^2}, $$
and there are $n+1$ cusps. Next, let $b_2(z)=z^{-1}+z^2$, so
$$ b_2\Bigl(e^{\pm\frac{\pi i}3}\Bigr)=b_2(-1)=0, $$
and $b_2(\bt)$ is not a Jordan curve. On the other hand, $b_2'(z)=2z-z^{-2}\not=0$ on $\bt$.

The following example is discussed in \cite[Example 4.4.3]{dav}. Let $q\ge2$, and
$$ b_3(z)=z^{q+1}+4z+z^{-q+1}, \qquad b_3\bigl(e^{i\th}\bigr)=2e^{i\th}(2+\cos q\th), $$
so $b_3(\bt)$ is a Jordan curve. Moreover,
$$ b_3'(\z)=\z^{-q}\bigl((q+1)\z^{2q}+4\z^q-(q-1)\bigr)\not=0, \quad \z\in\bt, $$
so $b_3$ is a LJ-polynomial. The example shows that sufficient condition \eqref{coef2} is not necessary for $b$ to be a LJ-polynomial.
\end{example}

\subsection{Regular Laurent polynomials}

Given a Laurent polynomial $b$ and $w\in\bc$, consider the algebraic equation, which plays a key role in what follows
\begin{equation}\label{maalg}
b(z)-w =\frac{P(z,w)}{z^m}=0, \quad P(z,w)=z^m\bigl(b(z)-w\bigr),
\end{equation}
$$ P(z,w) =b_{-m}+\ldots+(b_0-w)z^m+\ldots+b_kz^{m+k}=b_k\prod_{j=1}^{m+k}(z-z_j(w)). $$
Denote by $Z(w)=\{z_j(w)\}_{j=1}^{m+k}$ the zero divisor (\emph{i.e.,} roots with multiplicities) of $P$. Each such divisor splits naturally
in three parts: interior, exterior, and unimodular ones
$$ Z_{in}(w):=Z(w)\cap\bd, \quad Z_{ext}(w):=Z(w)\cap\bd_-, \quad Z_{un}(w):=Z(w)\cap\bt. $$
Under $|Z(w)|$ we mean the number of points in $Z(w)$, counting multiplicities, so $|Z(w)|=m+k$.

Denote by $\oo(b)$ an open set on the plane so that
\begin{equation}\label{wind} 
\oo(b):=\{w\in\bc\backslash b(\bt): \ {\rm wind}\,(b(t)-w)=0, \ \ t\in\bt\}, 
\end{equation}
or, in words, the winding number of the curve $b(\bt)$ around $w$ is zero.

\begin{lemma}\label{divuni}
We have:
\begin{enumerate}
\item[1.]  For any Laurent polynomial $b$ and for each $w\in\oo(b)$,
\begin{equation}\label{divout}
\begin{split}
Z_{in}(w) &=\{0<|z_1(w)|\le\ldots\le |z_m(w)|<1\}, \quad \qquad |Z_{in}(w)|=m, \\
Z_{ext}(w) &=\{1<|z_{m+1}(w)|\le\ldots\le |z_{m+k}(w)|<\infty\}, \ |Z_{ext}(w)|=k,
\end{split}
\end{equation}
and $Z_{un}(w)=\emptyset$.

\item[2.]  For each $w'\in b(\bt)$ and each LJ-polynomial $b$, there are positive constants $0<r_1<1<r_2$, $r_j=r_j(b)$ independent of $w'$, with
\begin{equation}\label{separ1}
Z_{in}(w')\subset\{|z|\le r_1<1\}, \quad Z_{ext}(w')\subset\{|z|\ge r_2>1\}, \quad w'\in b(\bt). 
\end{equation}
\item[3.] The function $|Z_{in}(w')|$ is constant on $b(\bt)$, and
\begin{equation}\label{numdivin} 
|Z_{in}(w')|=m-1 \ {\rm or} \ m, \quad |Z_{ext}(w')|=k \ {\rm or} \ k-1, \quad w'\in b(\bt). 
\end{equation}
\end{enumerate}
\end{lemma}

\begin{proof}
Clearly, $Z_{un}(w)=\emptyset$ for any $w\notin b(\bt)$. The first statement \eqref{divout} now follows directly from 
the Argument Principle, cf. \cite{dur}.

To prove \eqref{separ1}, note that $Z_{un}(w')=\{\l(w')\}$ is a single, simple (of multiplicity one) point on $\bt$. This is exactly the Jordan property
for Laurent polynomials, that is, there are no self-intersections and no cusps, respectively. When $w'$ traverses $b(\bt)$, the point $\l(w')$ traverses $\bt$. 
For each $w'\in b(\bt)$, the divisors $Z_{in}(w')$ and $Z_{ext}(w')$ are separated from the unit circle. 
Because of continuity of divisors and compactness of $b(\bt)$, \eqref{separ1} follows. 

Let us now shift $w'=w$ a bit into $\oo(b)$, while the point $\l(w')$ moves either inside $\bd$, or inside $\bd_-$ (it cannot
stay on the unit circle), and remains a simple point. So,
$$ |Z_{in}(w)|=|Z_{in}(w')|+1 \ {\rm or} \ |Z_{in}(w)|=|Z_{in}(w')|. $$
But, by \eqref{divout}, $|Z_{in}(w)|=m$, and the first equalities in \eqref{numdivin} follow. Since $|Z_{un}(w')|=1$ and $|Z(w')|=m+k$, the
second equalities in \eqref{numdivin} follow.
\end{proof}

\begin{definition}\label{defr}
A Laurent polynomial $b$ is called \emph{regular}, if either the interior divisor $Z_{in}(w)$ is contained strictly
inside $\bd$, or the exterior divisor $Z_{ext}(w)$ is contained strictly inside $\bd_-$, when $w$ varies over $\oo(b)$. 
In other words, there are two constants $0<r(b)<1<R(b)$, independent of $w$, so that either
\begin{equation}\label{reg1}
Z_{in}(w)\subset\{|z|\le r(b)<1\}, \qquad w\in\oo(b),
\end{equation}
or 
\begin{equation}\label{reg2}
Z_{ext}(w')\subset\{|z|\ge R(b)>1\}, \qquad w'\in\oo(b).
\end{equation} 
Equivalently, 
$$ \sup_{\oo(b)} |z_m(w)|\le r(b)<1 \quad {\rm or} \quad \inf_{\oo(b)} |z_{m+1}(w)|\ge R(b)>1. $$
\end{definition}

It is worth comparing \eqref{reg1}, \eqref{reg2} to \eqref{divout}, which holds for any Laurent polynomial, and to \eqref{separ1}, which holds for
any LJ-polynomial.

\begin{remark}\label{rem1}
For an arbitrary Laurent polynomial $b$ both conditions \eqref{reg1}, \eqref{reg2} hold automatically as long as $|w|$ is
large enough. Indeed, assume that
$$ |w|>2^{m+k}\|b\|_W=2^{m+k}\sum_{j=-m}^k |b_j|.  $$
Then
$$ w\,z_m^m(w)=\sum_{j=-m}^k b_jz_m^{m+j}(w), \qquad |w|\,|z_m(w)|^m\le \|b\|_W, $$
and by the assumption,
$$ 2^{m+k}\|b\|_W\,|z_m(w)|^m\le \|b\|_W, \quad (2|z_m(w)|)^m\le 2^{-k}\le1, \quad |z_m(w)|\le\frac12\,. $$
Similarly, we see that $|z_{m+1}(w)|\ge2$. Hence, given a Laurent polynomial, to check its regularity, it is sufficient to assume that
\begin{equation}\label{wbound}
|w|\le 2^{m+k}\|b\|_W=C_1.
\end{equation}
In the sequel $C_j$, $j=1,2,\ldots$, \ stay for positive constants which depend only on $b$. 
\end{remark}

\begin{theorem}\label{ljreg}
Each LJ-polynomial is regular.
\end{theorem}
\begin{proof}
According to Lemma \ref{divuni}, two situations may occur.

Assume first that $|Z_{in}(w')|=m$, $w'\in b(\bt)$. If $w'=w$ is shifted into a small strip $\oo_\ep:=\{w\in \oo(b): {\rm dist}(w,b(\bt))<\ep\}$, 
$\ep=\ep(b)$, the only point $\l(w')\in Z_{un}(w')$ moves inside $\bd_-$ (or, otherwise, $|Z_{in}(w)|=m+1$, that would contradict \eqref{divout}). 
So, due to continuity and compactness, there is a number $r_2=r_2(b)$, $r_1<r_2<1$ with
$$ Z_{in}(w)\subset\{|z|\le r_2\}, \qquad  w\in\oo_\ep. $$

Next, consider the set
$$ \widetilde\oo_\ep:=\{w\in\oo(b): {\rm dist}(w,b(\bt))\ge\ep, \ \ |w|\le C_1\}, $$
see \eqref{wbound}, and assume, on the contrary, that for a certain sequence $\{w_n\}$ in $\widetilde\oo_\ep$ one has
$$ |z_{m}(w_n)|\to 1. $$
Due to compactness of $\widetilde\oo_\ep$ we could let $w_n\to \tilde w\in\widetilde\oo_\ep$ and $z_m(w_n)\to \tilde z\in\bt$, 
so $\tilde z\in Z_{un}(\tilde w)$. But the latter is impossible since $\tilde w\notin b(\bt)$. Hence, for some $r_2<r_3<1$, depending on $b$,
$$ Z_{in}(w)\subset\{|z|\le r_3\}, \qquad  w\in\widetilde\oo_\ep. $$
For the rest see Remark \ref{rem1}.

2. Assume next, that $|Z_{in}(w')|=m-1$, $w'\in b(\bt)$. If $w'=w$ is shifted into a small strip $\oo_\ep':=\{w\in\oo(b): {\rm dist}(w,b(\bt))<\ep'\}$, 
$\ep'=\ep'(b)$, the only point $\l(w')\in Z_{un}(w')$ moves inside $\bd$ by \eqref{divout}. The above argument applies now to the exterior divisor,
and we come to \eqref{reg2}. The proof is complete.
\end{proof}

As it turns out, the class of regular polynomials is much wider than the class of LJ-polynomials.

Given two complex polynomials
$$ P(z)=\sum_{j=0}^n {n \choose j}\a_j z^j, \quad Q(z)=\sum_{j=0}^n {n \choose j}\b_j z^j, $$
of the same degree $n\ge1$, we say that $P$ and $Q$ are \emph{apolar} if
$$ \sum_{j=0}^n (-1)^j{n \choose j}\a_j\b_{n-j}=0. $$
Under a circular region we mean a closed or open half-plane, a disk or the exterior of a disk. The known
result of Grace \cite[Theorem 3.4.1]{rasch} states that given a circular region $C$ and apolar polynomials $P$ and $Q$, $P$ has at least one root in $C$
as long as $Z(Q)\subset C$.

\begin{proposition}
Let $b$ be a Laurent polynomial $\eqref{laur}$ with $m=1$. Assume that
\begin{equation}\label{vish}
|b_{k-s}|>{k+1 \choose s}\,|b_{-1}|
\end{equation}
for some $s=0,1,\ldots,k-1$. Then $b$ is a regular polynomial. In particular, the latter is true, when $|b_k|>|b_{-1}|$ (\emph{i.e.,}  $s=0$).
\end{proposition}
\begin{proof}
We put
\begin{equation*}
\begin{split}
P(z) &:=P(z,w)=b_{-1}+(b_0-w)z+\ldots+b_kz^{k+1}, \\
Q(z) &:=z^{k+1}+{k+1 \choose s} az^s, \quad a=(-1)^{k-s}\frac{b_{-1}}{b_{k-s}}\,,
\end{split}
\end{equation*}
and notice that $|b_{k-s}|>0$  under assumption \eqref{vish}. It is easy to check that, for each complex $w$, $P$ and $Q$ are apolar. Clearly,
$$ Z(Q)\subset C:=\{|z|\le\rho\}\subset\bd, \quad \rho=\left\{ {k+1 \choose s}\left|\frac{b_{-1}}{b_{k-s}}\right|\right\}^{\frac1{k+1-s}}<1. $$

As we know, in the case $m=1$, $Z_{in}(w)$ is a single, simple root $z_1(w)$ for $w\in\oo(b)$. 
So, by the Grace theorem, $z_1(w)\in C$, and \eqref{reg1} implies the regularity of $b$. The proof is complete.
\end{proof}

It is easy to construct a Laurent polynomial $b(z)=z^{-1}+b_kz^k$ with $|b_k|>1$ such that the curve $b(\bt)$ has self-intersection,
so $b$ is not a LJ-polynomial. Indeed, take $\z\in\bt$ with $|\z+1|<1$ and put
$$ b(z)=z^{-1}+b_2z^2, \quad b_2:=\frac1{\z(\z+1)}\,, \quad |b_2|>1. $$
Obviously, $b(\z)=b(1)$, as claimed.

There is another way to construct regular polynomials without the Jordan property. Indeed, for any regular polynomial $\b$ and an integer $m\ge2$
put $\b_m(z)=\b(z^m)$. The latter Laurent polynomial is still regular, but not a LJ-one.

\section{Linear growth of the resolvent}

The setting in the previous section was pure function theoretic. It is time now to bring in our main operator theoretic objects
under consideration: Toeplitz operators, resolvents and spectra, defined in Introduction.

%Each Laurent polynomial $b$ \eqref{laur} (more generally, each function from $L^\infty(\bt)$) is the symbol of a certain Toeplitz operator $T_b$, 
%with the functional realization
%$$  (T_b f)(t)=P_+ b(t)f(t), \qquad f\in H^2, $$
%$P_+$ is the orthogonal projection from $L^2(\bt)$ onto $H^2(\bt)$, and the matrix realization $T_b=\|b_{i-j}\|_{i,j\ge1}$ in $\ell^2(\bn)$, $b_n$ being
%the Fourier coefficients of the corresponding symbol. 
%By $\s(T_b)$, $\s_{ess}(T_b)$, we denote the spectrum and the essential spectrum of $T_b$, respectively.

By a theorem of Gohberg (see, e.g., \cite[Theorem 1.17]{bosil}), for each $b\in C(\bt)$ one has
$$ \s_{ess}(T_b)=b(\bt), \ \ \rho(T_b):=\bc\backslash\s(T_b)=\{\l\in\bc\backslash b(\bt): \ {\rm wind}(b(t)-\l)=0\}, $$
In particular, $\rho(T_b)=\oo(b)$, and we can freely replace the either of them by the other.  

The main result of the section concerns Toeplitz operators with special Laurent symbols and states that the LRG condition holds for such operators. 

%The latter means that the corresponding operator is given by a 
%banded Toeplitz matrix. The condition \eqref{linresgr} below is known as the {\sl Linear Resolvent Growth} condition (LRG), see \cite{nitr}.

\begin{theorem}\label{resol}
Let $b$ be a regular Laurent polynomial. Then
\begin{equation}\label{linresgr}
\|\bigl(T_b-w\bigr)^{-1}\|\le \frac{C(b)}{{\rm dist}\,\bigl(w,\s(T_b)\bigr)}\,, \qquad w\in\rho(T_b).
\end{equation}
In particular, the latter holds for LJ-polynomials.
\end{theorem}
\begin{proof}
Let $A$ be a bounded linear operator on a Hilbert space. The Neumann expansion for its resolvent is
$$ (A-w)^{-1}=-\frac1{w}\sum_{n\ge0} \lt(\frac{A}{w}\rt)^n, \qquad |w|>\|A\|, $$
so
$$ \|(A-w)^{-1}\|\le \frac1{|w|}\sum_{n\ge0}\left\|\frac{A}{w}\right\|^n=\frac1{|w|-\|A\|}\,, \quad |w|>\|A\|. $$
If $|w|\ge 2\|A\|$, then $|w|-\|A\|\ge\frac12 |w|$, and
$$ \|(A-w)^{-1}\|\le \frac2{|w|}\,, \qquad |w|\ge 2\|A\|. $$
Furthermore, for $z\in\s(A)$,
$$ {\rm dist}\,(w,\s(A))\le |w-z|\le |w|+|z|\le |w|+\|A\|\le\frac32 |w|, \ |w|\ge \frac23\,{\rm dist}\,(w,\s(A))   $$
and finally, with $A=T_b$, 
\begin{equation}\label{neum} 
\|\bigl(T_b-w\bigr)^{-1}\|\le \frac3{{\rm dist}\,(w,\s(T_b))}\,, \qquad |w|\ge 2\|T_b\|=2\|b\|_\infty=:C_2. 
\end{equation}
So, we assume in the sequel that $|w|\le C_2$.

Recall that
$$ P(z,w)=z^m\bigl(b(z)-w\bigr)=b_k\prod_{j=1}^{m+k} (z-z_j(w)), $$
$\{z_n(w)\}_{1}^{m+k}$ are labelled  as in \eqref{divout}.
The Wiener--Hopf factorization for the Laurent polynomial $b-w$ is
\begin{equation}\label{winhop}
b(z)-w=a_-(z,w)\,a_+(z,w), 
\end{equation}
where
\begin{equation}\label{winhop1}
a_-(z,w):=b_k\prod_{i=1}^m \Bigl(1-\frac{z_i(w)}{z}\Bigr), \quad a_+(z,w):=\prod_{j=1}^k(z-z_{m+j}(w)). 
\end{equation}

The bound for the resolvent relies upon the known theorem of Krein, see \cite[Theorem 1.15]{bosil},
\begin{equation}\label{reswh}
\bigl(T_b-w\bigr)^{-1}=T_{a_+^{-1}}T_{a_-^{-1}}, \quad \|\bigl(T_b-w\bigr)^{-1}\|\le \|a_+^{-1}\|_\infty\,\cdot\,\|a_-^{-1}\|_\infty.
\end{equation}

At the first stage we obtain the bounds for the factors on the RHS \eqref{reswh} {\it regardless of whether the regularity condition holds or not}.
Indeed, for $t\in\bt$,
$$ |a_+(t,w)| =\frac{|b(t)-w|}{|b_k|\,\prod_{i=1}^m|t-z_i(w)|}\ge \frac{|b(t)-w|}{|b_k|\,\prod_{i=1}^m(1+|z_i(w)|)}\ge\frac{|b(t)-w|}{2^m\,|b_k|}\,, $$
and so
\begin{equation}\label{term1}
\|a_+^{-1}\|_\infty\le\frac{2^m\,|b_k|}{{\rm dist}\,(w,b(\bt))}\,, \qquad w\in \widetilde\oo:=\oo(b) \cap \{|w|\le C_2\}.
\end{equation}

Going over to $a_-$, note that $w\in\widetilde\oo$ implies that all coefficients of the polynomial $P$ \eqref{maalg} are uniformly bounded 
(actually, only one coefficient of $P$, $b_0-w$, depends on $w$), and $b_{-m}b_k\not=0$. Hence, all the roots $z_n(w)$ are uniformly bounded 
for $w\in\widetilde\oo$
$$ \max_n |z_n(w)| \le C_3, \qquad |a_+(t,w)|\le \prod_{i=1}^k |t-z_{m+i}(w)|\le C_4, $$
and so
\begin{equation}\label{term2} 
|a_-^{-1}(t,w)|=\frac{|a_+(t,w)|}{|b(t)-w|}\,, \qquad \|a_-^{-1}\|_\infty\le\frac{C_4}{{\rm dist}\,(w,b(\bt))}\,, \qquad w\in\widetilde\oo. 
\end{equation}

At this point the regularity condition comes in, and we are able to refine either \eqref{term1} or \eqref{term2}.

1. Assume first that \eqref{reg1} holds, that is, 
$$ 1-|z_m(w)|>1-r(b)>0, \qquad w\in\widetilde\oo. $$ 
We have then
\begin{equation*}
|a_-(t,w)| =|b_k|\,\prod_{i=1}^m|t-z_i(w)|\ge |b_k|\,\prod_{i=1}^m(1-|z_i(w)|)\ge |b_k|(1-r(b))^m=C_5, 
\end{equation*}
and so
\begin{equation}\label{term1ref}
\|a_-^{-1}\|_\infty\le C_5^{-1}, \qquad w\in\widetilde\oo,
\end{equation}
which together with \eqref{term1} and \eqref{reswh} leads to LRG condition.

2. Assume next, that \eqref{reg2} holds, that is,
$$ |z_{m+1}(w)|-1>R(b)-1>0, \qquad w\in\widetilde\oo. $$
Then
\begin{equation*}
|a_+(t,w)| =\prod_{j=1}^k |t-z_{m+j}(w)|\ge \prod_{j=1}^k (|z_{m+j}(w)|-1)\ge (R(b)-1)^k=C_6>0.
\end{equation*}
Hence,
\begin{equation}\label{term2ref}
\|a_+^{-1}\|_\infty\le C_6^{-1}, \qquad w\in\widetilde\oo,
\end{equation}
which together with \eqref{term2} and \eqref{reswh} leads again to LRG condition. The proof is complete.
\end{proof}

Note that for LJ-polynomials the result is a particular case of \cite[Theorem 4]{pel}. Still, as we mentioned earlier, the class of regular 
polynomials, for which the LRG condition holds, is much wider than the class of LJ-polynomials.

\begin{example}\label{laujac}
Consider tridiagonal Toeplitz matrices with the symbol
$$ b(z)=\frac{b_{-1}}{z}+b_0+b_1z. $$
The image $b(\bt)$ of the unit circle is an ellipse on the plane as long as $|b_{-1}|\not=|b_1|$. So $b$ is the LJ-polynomial,
and LRG holds for the resolvent $(T_b-w)^{-1}$ on the resolvent set $\rho(T_b)$. The equality $|b_{-1}|=|b_1|$ means that
$b(\bt)$ is a closed interval on the plane traversed twice. Although $b$ is not a LJ-polynomial now, it is easy to check that
$$ b(t)=\a b_*(t)+\b, \quad b_*(t)=\ovl{at}+at,
$$ 
where $\a,\b, a\in\bc$. So $T_b=\a T_{b_*}+\b$, $T_{b_*}=T_{b_*}^*$ (cf. \cite{brha}). The latter implies that $T_b$ is a normal operator, so
$$ \|\bigl(T_b-w\bigr)^{-1}\|=\frac1{{\rm dist}\,(w,\s(T_b))}\,, \qquad w\in\rho(T_b). 
$$

Note also, that the simplest Laurent polynomial $b_0(z)=z+z^{-1}$ is irregular. Indeed, the roots $z_1, z_2$ are explicit, and
$$ 0<z_1(w)<1<z_2(w), \quad w\notin[-2,2], \quad \lim_{w\to 2+0} z_j(w)=1, \quad j=1,2, $$
as claimed. So, there is no direct relation between regularity and the LRG condition.
\end{example}

\section{Quadratic growth of the resolvent}\label{s3}

We say that the Quadratic Resolvent Growth (QRG) condition holds for the Toeplitz operator $T_b$ if
\begin{equation}\label{quadresgr}
\|(T_b-w)^{-1}\|\le \frac{C(b)}{{\rm dist}\,(w,\s(T_b))}\lp 1+\frac1{{\rm dist}\,(w,\s(T_b))}\rp, \ w\in\rho(T_b).
\end{equation}

Although we do not know whether the result of Theorem \ref{resol} holds for any Laurent polynomial, it is clear from
\eqref{reswh}, \eqref{term1} and \eqref{term2} (which have nothing to do with regularity), that QRG condition holds for an {\it arbitrary}
polynomial symbol $b$. Moreover, it is not hard to make sure, that the same method applies to an arbitrary {\it rational} symbol
and gives QRG condition, see \cite[Lemma 3.4]{gam}.

The main result of this section concerns the class of (in general, non-rational) symbols with infinitely many nonzero Fourier coefficients with
positive indices. On the other hand, we are able to prove the result under the strong assumption $m=1$. %The method is totally different. 
As compared to Theorem \ref{resol},  the key role is played now by the well-known formula which relates Toeplitz and Hankel matrices instead of the Wiener--Hopf factorization.

\begin{theorem}\label{infcoef}
Let $b$ be the symbol of the form
\begin{equation}\label{addas}
b(t)=\sum_{j=-1}^\infty b_jt^j, \qquad \b=\b(b):=\sum_{j\ge0}(j+1)|b_j|<\infty.
\end{equation}
Then the QRG condition $\eqref{quadresgr}$ holds for the Toeplitz operator $T_b$.
\end{theorem}
\begin{proof}
By the Wiener Theorem, the relation $w\notin b(\bt)$ implies that
$$ a(t)=a(t,w):=\frac1{b(t)-w}=\sum_{j=-\infty}^\infty a_j(w)t^j, \ \ \sum_{j=-\infty}^\infty |a_j(w)|<\infty, $$
belongs to the Wiener class, and the Toeplitz operator $T_a$ is well defined. Note, that, in general, $T_a\not=(T_b-w)^{-1}$.
Recall that $\rho(T_b)=\oo(b)$, so 
$${\rm wind}\,(b(t)-w)={\rm wind}\,a=0
$$
 for $w\in\rho(T_b)$, and, 
in particular, both operators $T_b-w$ and $T_a$ are invertible.

The main ingredient of the proof is the following well-known formula, which relates Toeplitz and Hankel matrices (see, e.g., \cite[Proposition 1.12]{bosil})
\begin{equation}\label{thform}
T_{uv}=T_u\,T_v+H_u\,H_{\tilde v}\,, \qquad u,v\in L^\infty(\bt), \quad \tilde v(t):=v(t^{-1}).
\end{equation}
Here $H_u$ is the Hankel matrix of a symbol $u$
\begin{equation}\label{hank} 
u(t)=\sum_{j=-\infty}^{\infty} u_jt^j, \qquad H_u=\|u_{i+j-1}\|_{i,j\ge1}.
\end{equation}
With $u=a$, $v=a^{-1}$, we come to
\begin{equation}\label{mainfor}
T_a\,T_{a^{-1}}=B:=I-H_a\,H_{\widetilde{a^{-1}}}.
\end{equation}
The operator on the LHS is invertible, and so is the one on the RHS. Hence
\begin{equation*}
T^{-1}_{a^{-1}}\,T^{-1}_a=B^{-1}, \qquad T^{-1}_{a^{-1}}=\bigl(T_b-w)^{-1}=B^{-1}\,T_a,
\end{equation*}
and we come to the following bound for the norm of the resolvent
\begin{equation}\label{normres1}
\Bigl\|\bigl(T_b-w\bigr)^{-1}\Bigr\|\le \bigl\|B^{-1}\bigr\|\,\|T_a\|.
\end{equation}
The last term equals
$$ \|T_a\|=\|a\|_\infty=\frac1{\min_t|b(t)-w|}=\frac1{{\rm dist}\,(w,b(\bt))}\le\frac1{{\rm dist}\,(w,\s(T_b))}\,, $$
so
\begin{equation}\label{normres2}
\Bigl\|\bigl(T_b-w\bigr)^{-1}\Bigr\|\le \frac{\bigl\|B^{-1}\bigr\|}{{\rm dist}\,(w,\s(T_b))}\,.
\end{equation}

To obtain the bound for the numerator on the RHS \eqref{normres2}, we need certain additional assumption on the symbol.

Since Hankel matrices \eqref{hank} are determined by the Fourier coefficients with positive indices, we have
$$ \dim H_{\widetilde{a^{-1}}}=1, \qquad H_{\widetilde{a^{-1}}}=\langle \cdot, e_1\rangle\,b_{-1}e_1, $$
and 
\begin{equation*}
\begin{split}
H_a\,H_{\widetilde{a^{-1}}} &=\langle \cdot, e_1\rangle\,\p, \quad \p:=b_{-1}\sum_{i=1}^\infty a_i e_i, \quad a_i=a_i(w), \\
B &=I-\langle \cdot, e_1\rangle\,\p. 
\end{split}
\end{equation*}
Since $B$ is invertible, it has no zero eigenvalue, so
$$ B\p=\p-\langle \p, e_1\rangle\,\p=(1-b_{-1}a_1)\p, \qquad 1-b_{-1}a_1\not=0. $$

It is a matter of a direct computation to verify that
$$ B^{-1}=I+\frac{\langle \cdot, e_1\rangle}{1-b_{-1}a_1}\,\p, $$
and we come to the bound
\begin{equation}\label{normbinv}
\bigl\|B^{-1}\bigr\|\le 1+\frac{\|\p\|}{|1-b_{-1}a_1|}\,.
\end{equation}

We proceed with the numerator in \eqref{normbinv}
$$ \|\p\|^2=|b_{-1}|^2\sum_{i=1}^\infty |a_i|^2\le |b_{-1}|^2\,\|a\|^2_{L^2(\bt)}=|b_{-1}|^2\,\int_{\bt}\frac{m(dt)}{|b(t)-w|^2}\,, $$
and hence
\begin{equation}\label{normphi}
\|\p\|\le \frac{|b_{-1}|}{{\rm dist}\,(w,\s(T_b))}\,.
\end{equation}

The bound for the denominator in \eqref{normbinv} is more complicated. Recall that
$$ P(z,w)=z(b(z)-w)=b_{-1}+(b_0-w)z+b_1z^2+\ldots, \quad |z|\le1, \quad b_{-1}\not=0, $$
so
$$ a(t,w)=\frac1{b(t)-w}=\frac{t}{P(t,w)}\,. $$
The function can be extended as a meromorphic function on the unit disk, continuous up to the unit circle. We also have $a(0)=0$, and this is the only
root of $a$ in the closed unit disk. As ${\rm wind}\,a=0$, by the Argument Principle, there is the only root $\z_0=\z_0(w)$ of $P$ there
$$ P(\z_0,w)=0, \qquad |\z_0|<1, \quad \z_0\not=0 \quad (P(0,w)=b_{-1}). $$

Next, for the Fourier coefficient $a_1$ we have
$$ a_1(w)=\int_{\bt} a(t,w)t^{-1}m(dt)=\int_{\bt} \frac{m(dt)}{P(t,w)}=\frac1{2\pi i}\,\int_{|\z|=1} \frac{d\z}{\z\,P(\z,w)}\,. $$
By the Residue Theorem,
\begin{equation*}
\begin{split}
a_1 &=\Bigl({\rm Res\, }_0+{\rm Res\, }_{\z_0}\Bigr)\,\frac1{\z\,P(\z,w)}=\frac1{P(0,w)}+\frac1{\z_0\,P'(\z_0,w)} \\
&=\frac1{b_{-1}}+\frac1{\z_0\,P'(\z_0,w)}. 
\end{split}
\end{equation*}
Hence,
\begin{equation*}
%\begin{split}
\frac1{1-b_{-1}a_1} =-\frac{\z_0\, P'(\z_0,w)}{b_{-1}}\,, \quad 
\frac1{|1-b_{-1}a_1|} \le \frac{|P'(\z_0,w)|}{|b_{-1}|}\,, \quad \z_0=\z_0(w)\in\bd.
%\end{split}
\end{equation*}

It remains only to estimate $|P'(\z_0,w)|$ in ``appropriate terms''.  Note that
$$ P'(\z_0,w)=(b_0-w)+2b_1\z_0+3b_2\z_0^2+\ldots, $$
and, although the coefficients $b_j$ do not depend on $w$, the root $\z_0$ does, and it can lie close enough to the unit circle. This is where the
assumption \eqref{addas} (\emph{i.e.,} the convergence of the given series) comes in. We have
$$ |P'(\z_0,w)|\le |w|+\b, \qquad \frac1{|1-b_{-1}a_1|} \le \frac{|w|+\b}{|b_{-1}|}\,, $$
and so, by \eqref{normbinv} and \eqref{normphi},
\begin{equation}
\bigl\|B^{-1}\bigr\|\le 1+\frac{|w|+\b}{{\rm dist}\,(w,\s(T_b))}\,.
\end{equation}

To complete the argument, we distinguish two cases.

1. Let $|w|\ge 2\|b\|_\infty$, then
$$ |b(t)-w|\ge |w|-\|b\|_\infty\ge\frac{|w|}2\,, \qquad {\rm dist}\,(w,\s(T_b))\ge \frac{|w|}2\,, $$
and hence, $\bigl\|B^{-1}\bigr\|\le C(b)$ for $w\in\oo$.

2. Let $|w|\le 2\|b\|_\infty$. Then
$$ \bigl\|B^{-1}\bigr\|\le C(b)\,\lp 1+\frac1{{\rm dist}\,(w,\s(T_b))}\rp. $$
Finally, in view of \eqref{normres2}, we come to the QRG condition \eqref{quadresgr}. The proof is complete.
\end{proof}

\noindent
{\bf Problem}. Find an \emph{optimal} function $\Phi$ so that the bound \eqref{upbou} for $T=T_b$ holds for all $b\in W$, the Wiener algebra.

\end{document}